\documentclass[12pt,a4paper,fleqn]{article}
\usepackage{amsmath,amssymb,euscript,amsthm,amsfonts}

\usepackage{exscale,amsthm,amssymb}
\usepackage{epsfig, graphicx}
\usepackage[english]{babel}
\usepackage{array}

\newtheorem{theorem}{Theorem}

\theoremstyle{definition}
\theoremstyle{remark}

\newcommand{\bol}[1]{\mbox{\boldmath$#1$}}
\newcommand{\bmu}{\bol{\mu}}
\newcommand{\bxi}{\bol{\xi}}

\newcommand{\bx}{\mathbf{x}}

\newcommand{\pxi}{\bxi_{\bx}(t_0,t)}
\newcommand{\pmu}{\bmu_{\bx}(t_0,t)}

\newcommand{\bI}{\mathbf{I}}
\newcommand{\hP}{\hat P}

\begin{document}
    \title{On the Generating Functional of the special case of $S$-Stopped Branching Processes}

    \author{Ostap Okhrin\thanks{C.A.S.E. - Center for Applied Statistics and Economics, Ladislaus von Bortkiewicz Chair of Statistics of Humboldt-Universit{\"a}t zu Berlin, Spandauer Stra{\ss}e $1$, D-$10178$ Berlin, Germany and Chair of Theoretical and Applied Statistics, Department of Mechanics and Mathematics, Ivan Franko National University of Lviv, Universytetska Str. 1, 79000, Lviv, Ukraine. Email: ostap.okhrin@wiwi.hu-berlin.de.}}

    \maketitle

    \begin{center}
        This paper is accepted for publication in \emph{Visn. L'viv. Univ., Ser. Mekh.-Mat.} (Bulletin of the
        Lviv University, Series in Mechanics and Mathematics).

    \end{center}

{\small \noindent \textbf{Abstract:}
    In this paper starting process with the infinite number of types of particles $\mu(t)$ generate stopped branching process $\xi(t)$, if by falling of the first one into the non empty set $S$ process stops. Here we consider the generating functional of the upper defined process.
}

\vspace{1cm}
To be more familiar with the general theory of branching processes, we recommend \cite{ll3} and \cite{ll4}. Let us consider the state spate of the types of particles $(X,\cal{A})$ with $\cal{A}$ being $\sigma$-algebra of Borel sets on  $X$, which contain all single-point sets. Let us define through $\Omega$ set for all nonnegative measures $\alpha$ on $\cal{X}$, that are concentrated on the finite subsets from $\mathcal{A}$ and can take integral values. Every element $\alpha\in\Omega$ can be characterized by a doubled vector $(x_1,n_1;\ldots;x_k,n_k)$ where $\{x_1,\ldots,x_k\}$ is that finite subset $\mathcal{A}$ on which $\alpha$ is concentrated. This assumption means, that in one particularly selected time point, only a finite subset of types is available. It is clear, that $n_i$ is a nonnegative integer, and corresponds to the number of particles of a specific type. Let us define by $\mathcal{Y}$ a Kolmogorov $\sigma$-algebra on $\Omega$, which is the smallest $\sigma$-algebra, that contains all cylindric sets $\{\alpha\in\Omega:\alpha(\{x\})=n\}$. On this space we consider an unbreakable Markov process with the transition probability $P(t_1,x,t_2,A)$, where $t_1,t_2\in\mathbb{R}$ is time, $x \in X,\;A\in \mathcal{Y}$. Considering every trajectory of the given process as an evolution of the movement of a particle, $P(t_1,x,t_2,A)$ can be interpreted as the probability of the event that a particle which started its movement from one point of a type $x \in X$, for time $t$ falls in the set $A \in \cal{Y}$. Every particle of a type $x$ has a random lifetime $\tau$. In the end of its life every particle of type $x$ promptly gives rise to a random number of offsprings, starting positions of which are distributed randomly on the space $\Omega$. Number and location of offsprings depend only on the location of the particle-ancestor at the moment of transformation. Later on every offspring evolutes analogously to and independently from other particles. Processes of similar art were considered in \cite{ll1}, \cite{ll2} and \cite{ll5}.

Let $\mu_{xt_0t}(A)$ be such a random measure, that for every $A \in \cal{Y}$ it is equal to a number of particles at time point $t$, that are falling into a set $A$, under condition that in the starting time point $t_0$ there was only one particle in point $x\in X$. By $\mu_{t_0t}(A)$ we define a random measure, which is equal to the number of particles at time point $t$, with types from the set $A$, without any knowledge about starting group of particles.

From here on we consider, that the space $X$ consists from an uncountable number of elements $X=\mathbb{R}^+$, this means, that the set of types is uncountable, or that to every type we put in line a nonnegative real number and vise versa. This is the special case of general branching processes.

Note, that total number of all particles in the beginning should be finite. This implies, that only a finite number of types corresponds to a nonzero number of particles. Although the total number of particles can be arbitrarily large.

Based on measure $\mu_{xt_0t}(A)$ we introduce a multivariate measure
$\bmu_{\bx t_0t}(A)$
\begin{equation*}
    \bmu_{\alpha(\bx) t_0t}(A)=\int_{X}\mu_{xt_0t}(A)\,dx,
\end{equation*}
where $\alpha\in \Omega$, $\bx \subset X$ is the set of types of particles, which is the argument for the function $\alpha$. In other words, if at the starting point $t_0$ one has a set of particles
$\alpha(x)=(\alpha_{x_1},\ldots, \alpha_{x_k})',\ \alpha_j\in\mathbb{N}$ with types in the set $\bx$, then for time $t$ one gets the set of points belonging to set $A$. Upper defined measure returns number of particles for each type from set $\bx$.

Having transition probabilities $P(t_1,x,t_2,A)$ let us introduce following probability $\widehat{P}(t_1,\alpha_1,t_2,\alpha_2)$, $\alpha_1,\alpha_2 \in
\Omega$, where $\widehat{P}$ is the probability saying, that if at time point $t_1$ one has set $\alpha_1$, then till time point $t_2$ one gets set $\alpha_2$. For the short hand writing let $\bmu_{\cdot}(t_0,t) = \bmu_{\cdot t_0t}(X)$.
It is obvious, that
\[\widehat{P}(t_1,\alpha_1,t_2,\alpha_2) = P\{\bmu_{\alpha_1}(t_1,t_2)=\alpha_2\}.\]
Let us fix the the finite subset $S\in\Omega$, $0\notin S$, which as the generalization can be of Lebesgue measure zero.
\emph{Stopped}, or \emph{$S$-stopped} multitype branching process is the process $\bxi_{\alpha
t}(X)$, defined for $t\in \mathbb{R}^+$ and $\alpha\in\Omega$ by equations
\begin{eqnarray*}
    \bxi_{\bx}(t_0,t) = \bxi_{\alpha t_0t}(X)=\left\{
    \begin{array}{rl}
        \bmu_{\bx}(t_0,t), &\; \mbox{if } \forall v, \; 0\leq v < t, \; \bmu_{\bx}(t_0,v) \notin S; \\
        \bmu_{\bx}(t_0,u), &\; \mbox{if } \forall v, \; 0\leq v < u, \;
        \bmu_{\bx}(t_0,v) \notin S, \; \\
        & \bmu_{\bx}(t_0,u) \in S, \; u<t.
    \end{array}
    \right.
\end{eqnarray*}
From this, for the $S$-stopped process $\pxi$, points of the set $S$ are additional states of absorption compared to the process $\pmu$. The latter had only one point of absorption $\{0\}$. In contrast to the process $\pmu$, in the $S$-stopped branching process $\pxi$ single particles in generation $t$ multiplies independently from each other following some probability law defined by the generated functional, only if $\pxi \notin S$. If the random vector $\pxi$ falls into the set $S$, the evolution of the process stops.

Let the set of all bounded $\cal{A}-$measurable functions on $X$ be defined by $\mathcal{F}$. Let us also define set of all positive and negative functions from $\mathcal{F}$ as $\mathcal{F}^+$ and $\mathcal{F}^-$ respectively. For all $f\in\mathcal{F},\ \alpha\in\Omega$ we write, that $[f,\alpha] = \int_X f(x)\alpha(dx)$. The shift operator $W_\alpha$ on $\Omega$ is defined by $W_\alpha A = \{\alpha':\alpha'-\alpha\in A\}$ for $A\subset\cal{Y}$.

Let us define separately transition probabilities for the $S$-stopped process as $\hP_S(t_1,\alpha,t_2, A)$, and for ordinary process as $\hP(t_1,\alpha,t_2, A)$. In both cases transition probabilities $\hP_S(t_1,\alpha,t_2, A)$ and $\hP(t_1,\alpha,t_2, A)$ are defined for all $t_1\leq t_2$, $\alpha\in\Omega$ and $ A\subset\mathcal{Y}$ define branching processes with continuous time, if they fulfill special properties. Both transition probabilities $\hP(t_1,\alpha,t_2, A)$ and $\hP_S(t_1,\alpha,t_2, A)$ are $\mathcal{Y}$-measurable functions, as functions on $\alpha$, and also nonnegative measures on $\mathcal{Y}$, as functions on $A$. It is clear, that it does not matter from which set does the process starts, because it will always fall into the space of all events $\Omega$. Thus $\hP(t_1,\alpha,t_2,\Omega)=\hP_S(t_1,\alpha,t_2,\Omega)=1$. Following the same logic, for the zero time interval process can fall into some set only if it was located in this set in the beginning of this time interval, thus $\hP(t_1,\alpha,t_1, A)=\hP_S(t_1,\alpha,t_1, A)=\bI\{\alpha\in A\}$. It is clear, that a classical Kolmogorov-Chapman equation should hold for the ordinary process
\begin{equation}\label{eqn:kol_chapman}
    \hP(t_1,\alpha,t_3, A) = \int_\Omega \hP(t_2,\alpha',t_3, A)\hP(t_1,\alpha,t_2,d\alpha'),\ \forall t_1\leq t_2\leq t_3,
\end{equation}
and a small modification for the $S$-stopped
\begin{equation}\label{eqn:s_kol_chapman}
    \hP_S(t_1,\alpha,t_3, A) = \int_{\Omega\setminus (S\setminus A)} \hP(t_2,\alpha',t_3, A)\hP(t_1,\alpha,t_2,d\alpha'),\ \forall t_1\leq t_2\leq t_3.
\end{equation}
As all particles in the process evolute independently from each other, we may write following relationships
\begin{eqnarray}\nonumber
    \hP(t_1,\alpha_1+\alpha_2,t_2, A) &=& \int_\Omega\int_\Omega \bI\{\alpha_1'+\alpha_2'\in A\}\hP(t_1,\alpha_1,t_2,d\alpha_1')\\\label{eqn:independent_part}
    &\times&\hP(t_1,\alpha_2,t_2,d\alpha_2'),\\
    \nonumber
    \hP_S(t_1,\alpha_1+\alpha_2,t_2, A) &=& \int_{\Omega\setminus (S\setminus A)}\int_{\Omega\setminus (S\setminus A)} \bI\{\alpha_1'+\alpha_2'\in A\}\hP(t_1,\alpha_1,t_2,d\alpha_1')\\\label{eqn:s_independent_part}
    &\times&\hP(t_1,\alpha_2,t_2,d\alpha_2').
\end{eqnarray}
Under the given conditions, for the process with the continuous time, it is natural to assume, that for $\triangle\rightarrow 0$ both processes are equal $\hP(t,\alpha,t+\triangle, A) = \hP_S(t,\alpha,t+\triangle, A)$ and
\begin{eqnarray}\label{eqn:transition_approx}
    \hP(t,\alpha,t+\triangle, A) = \left\{\begin{array}{ll}
    1 + p(t, \alpha,  A)t + o(t), & \mbox{for } p(t, \alpha,  A)<0,\ \alpha \in  A;\\
    p(t, \alpha,  A)t + o(t), & \mbox{for } p(t,\alpha,  A)\geq0,\ \alpha \notin  A.\\
    \end{array}\right.
\end{eqnarray}
Thus upper defined assumption can be reformulated as follows
\begin{eqnarray*}
    \hP(t_1,\alpha,t_2, A) &=& p(t_1,\alpha, A)(t_2-t_1) + o(t_2-t_1),\ \\
    & &\quad\quad t_2\rightarrow t_1^-,\ \alpha\notin A,\\
    \frac{\hP(t_1,\alpha,t_2, A) - \hP(t_1,\alpha,t_1, A)}{t_2-t_1} &=& p(t_1,\alpha, A) + o(1),\\
    \lim_{t_2\rightarrow t_1^+}\frac{\hP(t_1,\alpha,t_2, A) - \hP(t_1,\alpha,t_1, A)}{t_2-t_1} &=& p(t_1,\alpha, A).
\end{eqnarray*}
Similarly one can prove this for the right limit
\begin{equation*}
\lim_{t_2\rightarrow t_1^-}\frac{\hP(t_2,\alpha,t_1, A) - \hP(t_1,\alpha,t_1, A)}{t_2-t_1} = p(t_1,\alpha, A).
\end{equation*}
This is equivalent to the fact, that
\begin{equation*}
    \frac{\partial}{\partial t}\hP(t,\alpha,t, A) = p(t,\alpha, A).
\end{equation*}
Function $p(t,\alpha, A)$ will be called a transition density of the branching process. It can be also proved similarly to the case when $\alpha\in A$. From the properties of transition probabilities $\hP(t_1,\alpha, t_2, A)$ and $\hP_S(t_1,\alpha, t_2, A)$ we get following conditions on $p(t, \alpha,  A)$. $p(t, \alpha,  A)$ exists and is finite for all $t,\,\alpha,\, A$, is $\mathcal{Y}$-measures, as a function on $\alpha$ and is generalized measure on $\mathcal{Y}$ as function on $ A$. Is is also clear, that $p(t, \alpha,\{\alpha\})\leq0$, $p(t, \alpha, A\subset\Omega-\{\alpha\})\geq0$ and $p(t, \alpha, \Omega)=0$.

Let us introduce ordinal and logarithmic Laplace functional for both processes, based on the main transition probabilities
\begin{eqnarray*}
    F(t_1,\alpha, t_2, s) &=& \int_\Omega \exp\left\{\int_X s(x)\alpha'(dx)\right\}\hP(t_1,\alpha, t_2,d\alpha')\\
        &=& E_{\hP} \exp[s, \alpha],\\
    \Psi(t_1,\alpha, t_2, s) &=& \log F(t_1,\alpha, t_2, s) = \log E_{\hP} \exp[s, \alpha],\\
    F_S(t_1,\alpha, t_2, s) &=& \int_\Omega \exp\left\{\int_X s(x)\alpha'(dx)\right\}\hP_S(t_1,\alpha, t_2,d\alpha')\\
        &=& E_{\hP_S} \exp[s, \alpha]\\
        &=& \int_{\Omega\setminus(S\setminus A)} \exp\left\{\int_X s(x)\alpha'(dx)\right\}\hP_S(t_1,\alpha, t_2,d\alpha'),\\
    \Psi_S(t_1,\alpha, t_2, s) &=& \log F_S(t_1,\alpha, t_2, s) = \log E_{\hP_S} \exp[s, \alpha].
\end{eqnarray*}
In the definition of the functional for the stopped process one can integrate over the whole space, because the falling into the absorption set $S$ is controlled by the transition probability $\hP_S(t_1,\alpha, t_2,d\alpha')$, but for convenience and consistency with other notations we will narrow the set. Let us also introduce corresponding functionals based on the transition densities $p(t,\alpha,  A)$
\begin{eqnarray*}
    \phi(t_1,\alpha, s) &=& \int_\Omega \exp\left\{\int_X s(x)\alpha'(dx)\right\}p(t_1,\alpha, d\alpha')\\
        &=& E_{p} \exp[s, \alpha]\\
    \psi(t_1,\alpha, s) &=& \exp\{-[s, \alpha]\}\phi(t_1,\alpha, s)\\
    \phi_S(t_1,\alpha, s) &=& \int_{\Omega\setminus(S\setminus A)} \exp\left\{\int_X s(x)\alpha'(dx)\right\}p(t_1,\alpha, d\alpha')\\
    \psi_S(t_1,\alpha, s) &=& \exp\{-[s, \alpha]\}\phi_S(t_1,\alpha, s)
\end{eqnarray*}
From the independency of the reproduction of particles (\ref{eqn:independent_part}) it holds that
\begin{equation*}
    F_*(t_1,\alpha_1+\alpha_2, t_2, A) = F_*(t_1,\alpha_2, t_2, A)\cdot F_*(t_1,\alpha_2, t_2, A),
\end{equation*}
where ``$*$'' means, that written above holds for both processes. (\ref{eqn:transition_approx}) implies that
\begin{equation*}
    \frac{\partial}{\partial t_1}F_*(t_1,\alpha, t_2, s)|_{t_1=t_2^-}=\phi_*(t_1,\alpha,s).
\end{equation*}
Hence
\begin{equation*}
    \phi_*(t,\alpha_1+\alpha_2,s)= \phi_*(t,\alpha_1,s)\exp[s,\alpha_2] + \phi_*(t,\alpha_2,s)\exp[s,\alpha_1].
\end{equation*}
Dividing both sides of the last equality by $\exp[s,\alpha_1+\alpha_2]$ we get
\begin{equation*}
    \psi_*(t,\alpha_1+\alpha_2,s)= \psi_*(t,\alpha_1,s) + \psi_*(t,\alpha_2,s).
\end{equation*}
In contrary to earlier works on $S$-stopped branching in this work we will be following classical assumption, that the process starts from one particle. Obviously, most of propositions we can write in the point-wise form, like
\begin{equation*}
    \psi_*(t,\alpha, s) = \int_X \psi_*(t,x,s)\alpha(dx) = [\psi_*(t,\cdot,s),\alpha],\ \alpha\in\mathcal{Y},\ x\in X,
\end{equation*}
thus
\begin{equation*}
    p(t,\alpha,A) = \sum_{i=1}^kn_ip(t,x_i,W_{x_i-\alpha},A),
\end{equation*}
where $\alpha = \{n_1,x_1;\ldots;n_k,x_k\}$. In this case we actually are taking into account the probability of transition from one point of the fixed type and multiply by the number of points of this type. We do this for all types on which function $\alpha$ is concentrated.

As most of the papers are describing the probability of extinction of the branching processes under different assumptions, in our case the probability of extinction of the general branching process without ``stopping'' conditions is given by  $\hP(t_1,x,t_2,\{0\})$. If this probability converges to 1 for all $x\in X$, then we call our process an extincting process. The probability of extinction of the $S$-stopped branching process will be defined through $\hP_S(t_1,x,t_2,S\cup \{0\})$.

If all assumptions concerning transition densities are fulfilled, then we can construct a fundamental Feller solution for ordinary and $S$-stopped branching processes. Let us derive all the theory for the ordinary processes, and then overtake it on the stopped ones. For this we need few more definitions.

Let us define by $q(t,\alpha)$ the probability that the process in the neighborhood of time $t$ remains on the place, but not point-wise. For example, if we have one point $a$ of type $x_1$ and one point $b$ of type $x_2$, than it is allowed, that that point $a$ of type $x_1$ is moved into a point $c$ of type $x_2$ and point $b$ of type $x_2$ is moved into a point $d$ of type $x_1$. In this case the process has been change point-wise, but in generals remains in the same state $\alpha$, as was in the beginning, because in the end we get one point of type $x_1$ and one point of type $x_2$ respectively. Thus upper defined probability can be rewritten as
\begin{equation*}
    q(t, \alpha) = p(t,\alpha, \alpha).
\end{equation*}
Let $p_1(t, \alpha, A)$ be the probability, that the process in the neighborhood of $t$ moved from its state into the set $A\setminus\{\alpha\}$
\begin{equation*}
    p_1(t,\alpha, A) = p(t, \alpha, A\setminus\{\alpha\}).
\end{equation*}
The probability, that the process over time interval $[t_1,t_2]$ remains not point-wise in the original state we define through
\begin{equation*}
    J(t_1,t_2,\alpha) = \int_{t_1}^{t_2}p(t,\alpha, \alpha)dt.
\end{equation*}
Let
\begin{eqnarray*}
    \hP^{(0)}(t_1,\alpha,t_2,A) &=& \bI\{\alpha\in A\}\exp\left\{\int_X J(t_1,t_2,x)\alpha(dx)\right\}\\
        &\approx& \bI\{\alpha\in A\}\exp\left\{\sum_{i=1}^k J(t_1,t_2,x_i)n_i\right\}
\end{eqnarray*}
be the probability, that the process for the whole time interval $[t_1,t_2]$ remains point-wise on the original state, what means, that there were no changes inside the process, and every particle at every time point has been moved into itself. Later we define an independent series of events and their probabilities
\begin{equation*}
    \hP^{(k)}(t_1,\alpha,t_2,A) = \int_{t_1}^{t_2}\exp[J(t_1,t,\cdot),\alpha]\int_\Omega \hP^{(k-1)}(t,\alpha',t_2,A)p_1(t,\alpha, d\alpha')dt,
\end{equation*}
what for every $k$ means $k$ point-wise changes inside the process before falling into $A$. For the ordinary process we have the Feller solution
\begin{equation*}
    \hP(t_1,\alpha, t_2,A) = \sum_{k=0}^\infty \hP^{(k)}(t_1,\alpha, t_2,A).
\end{equation*}
This holds only for ordinary non-stopped processes. For $S$-stopped processes we introduce several notations of probabilities linked with the falling or not-falling into a set $S$. Especially the probability, that process falls into  $A$ by not falling into itself $\{\alpha\}$ and not falling into absorbtion set $S$
\begin{equation*}
    p_2(t, \alpha,A) = p(t, \alpha, A\setminus\{\alpha\}\setminus S).
\end{equation*}
The probability, that the process falls into $A\cap S$ we define through
\begin{equation*}
    p_S(t, \alpha,A) = p(t, \alpha, A\cap S),
\end{equation*}
and the probability, that process falls into $A\setminus S$ by
\begin{equation*}
    p_{\bar S}(t, \alpha,A) = p(t, \alpha, A\setminus S).
\end{equation*}
At first, let us consider the case, when $A\cap S\ne \emptyset$ and process falls into $A\setminus S$. In this situation let us define
\begin{eqnarray*}
    \hP_{\bar S}^{(k)}(t_1,\alpha, t_2, A) &=& \hP^{(k)}(t_1,\alpha, t_2, A\setminus S) \\&=& \int_{t_1}^{t_2}\exp[J(t_1,t,\cdot),\alpha]\int_\Omega \hP^{(k-1)}(t,\alpha',t_2,A)p_2(t,\alpha, d\alpha')dt.
\end{eqnarray*}
Let us introduce following series of probabilities of independent events, which correspond to falling of the process in the set $A$, which is not necessary disjunct with $S$, but under the condition, that for every intermediate fall of the process into $S$, process stops
\begin{eqnarray*}
    \hP^{(0)}_S(t_1,\alpha, t_2, A) &=& \hP^{(0)}(t_1,\alpha, t_2, A), \\
    \hP^{(k)}_S(t_1,\alpha, t_2, A) &=& \hP_{\bar S}^{(k)}(t_1,\alpha, t_2, A) + \int_{t_1}^{t_2}\int_\Omega p_S(t,\alpha',A)\hP_{\bar S}^{(k-1)}(t_1,\alpha, t,d\alpha')dt.
\end{eqnarray*}
Based on these probabilities, Feller solution for $S$-stopped branching processes is given by
\begin{equation*}
    \hP_S(t_1,\alpha, t_2,A) = \sum_{k=0}^\infty \hP^{(k)}_S(t_1,\alpha, t_2,A).
\end{equation*}
Further aim of our analysis is to consider the link between functional equation for general and $S$-stopped branching processed.
\begin{theorem}Functional equation for $S$-stopped branching processed became
\[\frac{\partial}{\partial s}F(s, w, t, f) = -\int_{\Omega}F(s,w',t,f)p_{\bar S}(s,w,dw') + \mathcal{B}\]
where $\mathcal{B} = \partial B/\partial s$, and $B$ from (\ref{eqn:funk_rnja}) in the theorem.
\end{theorem}
\begin{proof}
\begin{eqnarray*}
    F_S(s, w, t, f) &=& \int_\Omega \exp[f, w']\hP_S(s, w, t, dw')\\
    &=& \sum_{k=0}^\infty \int_\Omega \exp[f, w']\hP_S^{(k)}(s, w, t, dw')\\
    &=& \sum_{k=0}^\infty \int_\Omega \exp[f, w']P_{\bar S}^{(k)}(s, w, t, dw')\\
    &+& \sum_{k=0}^\infty \int_\Omega \exp[f, w']  \int_{s}^{t}\int_\Omega p_S(t',\alpha',dw')\hP_{\bar S}^{(k)}(s,w, t',d\alpha')dt'\\
    &=& \sum_{k=0}^\infty \int_\Omega \exp[f, w']P^{(k)}(s, w, t, dw'\setminus S) + B\\
    &=& A+B.
\end{eqnarray*}
Let us consider both terms $A$ and $B$ separately, at first the first one
\begin{align*}
    A =& \sum_{k=0}^\infty \int_\Omega \exp[f, w']P^{(k)}(s, w, t, dw'\setminus S)
      = \sum_{k=0}^\infty \int_{\Omega\setminus S} \exp[f, w']P^{(k)}(s, w, t, dw')\\
      =& \int_{\Omega\setminus S}\bI\{w'\in w\}\exp\left\{\int_X f(x)w'(dx)+\int_X J(s,t,x)w'(dx)\right\}p_1(s,w,dw')\\
      +& \sum_{k=1}^\infty \int_{\Omega\setminus S}\int_s^t \exp[f+J(s, s',\cdot), w']\int_\Omega P^{(k-1)}(s', w'', t, dw')p_1(s',w,dw'')ds'\\
      =& \int_{\Omega\setminus S}\bI\{w'\in w\}\exp[f+J(s,t,\cdot),w']p_1(s,w,dw')\\
      +& \int_{\Omega\setminus S}\int_s^t \exp[f+J(s, s',\cdot), w']\sum_{k=1}^\infty\int_\Omega P^{(k-1)}(s', w'', t, dw')p_1(s',w,dw'')ds'\\
      =& \bI\{w'\in \Omega\setminus S\}\exp[f+J(s,t,\cdot),w]\\
      +& \int_{\Omega\setminus S}\int_s^t \exp[J(s, s',\cdot), w']\int_\Omega \exp[f, w'] \sum_{k=0}^\infty P^{(k)}(s', w'', t, dw')p_1(s',w,dw'')ds'\\
      =& \bI\{w\in \Omega\setminus S\}\exp[f+J(s,t,\cdot),w]\\
      +& \int_{\Omega\setminus S}\int_s^t \exp[J(s, s',\cdot), w]F(s',w',t,f)p_1(s',w,dw')ds'.
\end{align*}
We will not shape the second term to the same good form, but just show that one can avoid the endless sums and recursions
\begin{align}\nonumber
    B =& \sum_{k=0}^\infty \int_\Omega \exp[f, w']  \int_{s}^{t}\int_\Omega p_S(t',\alpha',dw')\hP_{\bar S}^{(k)}(s,w, t',d\alpha')dt'\\\nonumber
      =& \sum_{k=0}^\infty \int_\Omega \exp[f, w'] \int_{s}^{t}\int_{\Omega\setminus S} p_S(t',\alpha',dw')\hP^{(k)}(s,w, t',d\alpha')dt'\\\nonumber
      =& \int_\Omega \exp[f, w']\int_s^t\int_{\Omega\setminus S}p_S(t',\alpha',dw')\bI\{w\in d\alpha'\}\exp[J(s, t',\cdot),w]dt'\\\nonumber
      +& \sum_{k=0}^\infty \int_\Omega \exp[f, w']\int_s^t\int_{\Omega\setminus S}p_S(t',\alpha',dw')\int_s^{t'}\exp[J(s,t'',\cdot),w]\\\nonumber
      \times&\int_\Omega \hP^{(k)}(t'',\alpha'',t',d\alpha')p_1(t'',w,d\alpha'')dt''dt'\\\label{eqn:funk_rnja}
      =& \int_\Omega\exp[f, w']\int_s^tp_S(t',w,dw')\exp[J(s,t',\cdot),w]dt'\\\nonumber
      +& \int_\Omega\exp[f, w']\int_s^t\int_{\Omega\setminus S}p_S(t',\alpha',dw')\int_s^{t'}\exp\{[J(s, t'',\cdot),w]-[f, \alpha']\}\\\nonumber
      \times&F(t'',\alpha'',t',f)p_1(t'',w,d\alpha'')dt''dt'.
\end{align}
Part $B$ does not contain any recursions and infinite sums, and is also continuous and differentiable with respect to $s$ what implies, that $\exists\ \partial B/\partial s=:\mathcal{B}$. For the following calculations let us find
\begin{equation*}
    \frac{\partial J(s,t,\cdot)}{\partial s} = \frac{\partial}{\partial s}\int_s^t p(s',w,w)ds' = -p(s,w,w)=-q(s,w).
\end{equation*}
Let us calculate the derivative with respect to $s$ from the whole functional
\begin{align*}
    \frac{\partial}{\partial s}F(s, w, t, f) =& \frac{\partial A}{\partial s} + \mathcal{B}\\
      =& \frac{\partial}{\partial s}\bI\{w\in \Omega\setminus S\}\exp[f+J(s,t,\cdot),w]\\
      +& \frac{\partial}{\partial s}\int_{\Omega\setminus S}\int_s^t \exp[J(s, s',\cdot), w]F(s',w',t,f)p_1(s',w,dw')ds' + \mathcal{B}\\
      =& \frac{J(s,t,\cdot)}{\partial s}\bI\{w\in \Omega\setminus S\}\exp[f+J(s,t,\cdot),w]\\
      -&\int_{\Omega\setminus S}\exp[J(s,s,\cdot),w]F(s,w',t,f)p_1(s,w,dw')\\
      +&\frac{\partial J(s,t,\cdot)}{\partial s}\int_{\Omega\setminus S}\int_s^t \exp[J(s,s',\cdot),w]F(s',w',t,f)p_1(s',w,dw')ds' + \mathcal{B}\\
      =& -q(s, w)\exp[f+J(s,t,\cdot),w]\bI(w\in \Omega\setminus S)\\
      -&\int_{\Omega\setminus S}F(s,w',t,f)p_1(s,w,dw')\\
      -& q(s,q)\int_s^t\int_{\Omega\setminus S}\exp[J(s,s',\cdot),w]F(s',w',t,f)p_1(s',w,dw')ds' + \mathcal{B}\\
      =& -\int_{\Omega\setminus S}F(s,w',t,f)p(s,w,dw') + \mathcal{B}\\
      =& -\int_{\Omega}F(s,w',t,f)p_{\bar S}(s,w,dw') + \mathcal{B}.
\end{align*}
\end{proof}
In the case of non $S$ stopped process, i.e. $S = \emptyset$, our results coincides with \cite{ll13}. In that case $p_S(t, \alpha, A) = p(t, \alpha, A\cap S) = p(t, \alpha, A\cap \emptyset) = p(t, \alpha, \emptyset) = 0,$ $p_{\overline{S}}(t, \alpha, A) = p(t, \alpha, A \setminus S) = p(t, \alpha, A)$. This implies that $B = 0$, then $\mathcal{B} = 0$, and
\begin{equation*}
    -\int_{\Omega}F(s,w',t,f)p_{\bar S}(s,w,dw') = -\int_{\Omega}F(s,w',t,f)p(s,w,dw').
\end{equation*}
This implies that for $S = \emptyset$
\begin{equation*}
    \frac{\partial}{\partial s}F(s, w, t, f) = -\int_{\Omega}F(s,w',t,f)p(s,w,dw').
\end{equation*}
What is equivalent to the equation (2.23) in \cite{ll13}. Similarly, it can be reduced to (2.24) in \cite{ll13}.

\end{document}